\DeclareSymbolFont{AMSb}{U}{msb}{m}{n}
\documentclass[11pt,a4paper,twoside,leqno,noamsfonts]{amsart}
\usepackage[a4paper,top=3.5cm,bottom=3.1cm,left=3.4cm,right=3.4cm]{geometry}

\usepackage[english]{babel} 
\usepackage[utopia]{mathdesign}
\DeclareSymbolFont{usualmathcal}{OMS}{cmsy}{m}{n}
\DeclareSymbolFontAlphabet{\mathcal}{usualmathcal}

\usepackage{indentfirst}
\usepackage[colorlinks,bookmarks]{hyperref} 
\usepackage{braket}
\usepackage{caption}
\usepackage{stmaryrd}
\usepackage{comment}
\usepackage{multirow}
\usepackage{booktabs}
\usepackage[all]{xy}

\usepackage{listings}
\usepackage{xcolor}

\usepackage{microtype}

\newtheorem{cor}{Corollary}[section]
\newtheorem{defn}[cor]{Definition}
\newtheorem{lem}[cor]{Lemma}
\newtheorem{thm}[cor]{Theorem}

\newtheorem{prop}[cor]{Proposition}

\newtheorem{problem}[cor]{Problem}
\newtheorem{question}[cor]{Question}

\definecolor{commentgreen}{RGB}{2,112,10}
\definecolor{eminence}{RGB}{108,48,130}
\definecolor{frenchplum}{RGB}{149,20,83}

\lstdefinelanguage{M2}
{
keywords={dim, singularlocus, saturate, ideal },
emph={random, mingens, matrix, map, det},
sensitive=false,
morecomment=[l]{--},
morecomment=[s]{/*}{*/},
morestring=[b]",
}

\lstnewenvironment{m2_code}[1][]
{\lstset{basicstyle=\tiny\ttfamily, columns=fullflexible,
language=M2,
keywordstyle=\color{frenchplum},
emphstyle={\color{eminence}},
commentstyle=\color{commentgreen},
numbers=none, numberstyle=\tiny,
stepnumber=2, numbersep=5pt, frame=single, #1}}{}

\newcommand{\bP}{\mathbb{P}}
\newcommand{\bC}{\mathbb{C}}
\newcommand{\bG}{\mathbb{G}}
\newcommand{\bL}{\mathbb{L}}
\newcommand{\bA}{\mathbb{A}}
\newcommand{\bD}{\mathbb{D}}
\newcommand{\bQ}{\mathbb{Q}}
\newcommand{\bZ}{\mathbb{Z}}

\newcommand{\sB}{\mathscr{B}}
\newcommand{\sC}{\mathscr{C}}
\newcommand{\sF}{\mathscr{F}}
\newcommand{\sH}{\mathscr{H}}
\newcommand{\sO}{\mathscr{O}}
\newcommand{\sS}{\mathscr{S}}
\newcommand{\sU}{\mathscr{U}}

\newcommand{\Spec}{\operatorname{Spec}}
\newcommand{\Var}{\operatorname{Var}}
\newcommand{\Pic}{\operatorname{Pic}}
\newcommand{\br}{\operatorname{Br}}
\newcommand{\coh}{\operatorname{Coh}}
\newcommand{\id}{\operatorname{Id}}
\renewcommand{\hom}{\operatorname{Hom}}

\title{Equivalence of K3 surfaces from Verra threefolds}
\author[G.~Kapustka]{Grzegorz Kapustka}
\address{Department of Mathematics and Informatics,
Jagiellonian University, {\L}ojasiewicza 6, 30-348 Krak{\'o}w, Poland.}
\email{grzegorz.kapustka@uj.edu.pl}
\author[M.~Kapustka]{Micha\l{} Kapustka}
\address{University of  Stavanger, Department of Mathematics and Natural Sciences, NO-4036 Stavanger, Norway and Institute of Mathematics of the Polish Academy of Sciences, ul. \'Sniadeckich 8, 00-656, Warszawa, Poland}
\email{michal.kapustka@uis.no}
\author[R.~Moschetti]{Riccardo Moschetti}
\address{University of  Stavanger, Department of Mathematics and Natural Sciences, NO-4036 Stavanger, Norway}
\email{riccardo.moschetti@uis.no}

    
\begin{document}
\begin{abstract}
We study $(2,2)$ divisors in $\bP^2\times \bP^2$ giving rise to pairs of non-isomorphic, derived equivalent and $\bL$-equivalent K3 surfaces of degree $2$. In particular, we confirm the existence of such fourfolds as predicted by Kuznetsov and Shinder in \cite{KS}.
\end{abstract}
\maketitle
\section{Introduction}
Let $K_0(\Var/k)$ be the Grothendieck ring of varieties over a field $k$. It is the free abelian group generated
by isomorphism classes $[X]$ of algebraic $k$-varieties $X$, with relations 
$$[X]=[Z]+[U]$$
for every closed subvariety $Z\subset X$ with an open complement $U \subset X$. 
The product structure which makes $K_0(\Var/k)$ a ring is induced by products of varieties, namely
$$[X] \cdot [Y] = [X \times_k Y].$$
The unity is $[\Spec(k)]$, the class of a point. We will denote by $\bL$ the class of the affine line $[\bA^1]$, also called the Lefschetz motive. It is not a surprise that $K_0(\Var/k)$ encodes lots of geometrical relations: a remarkable example is the connection between zero-divisors in the Grothendieck ring and the rationality of cubic fourfolds, discussed in \cite{GS}. We consider the following two equivalence relations between varieties over $k$. 

\begin{defn}
Two varieties $X$ and $Y$ over $k$ are said to be 
\begin{itemize}
\item $\bL$-equivalent if\: $\bigl([X] -[Y]\bigr) \cdot \bL^r = 0$ in $K_0(\Var/k)$ for some integer $r \geq 0$.
\item $\bD$-equivalent if the bounded derived categories of coherent sheaves $\bD^b(X)$ and $\bD^b(Y)$ are equivalent.
\end{itemize}
\end{defn}

Notice that $\bL$-equivalence does not imply $\bD$-equivalence, an example is provided by $\bP^1 \times \bP^1$ and the blowup of $\bP^2$ in one point: while the motivic class of both is $\bL^2+2\bL+1$. These are non-isomorphic Fano varieties, which can not be derived equivalent by  \cite[Theorem 2.5]{MR1818984}.

This work is motivated by the opposite problem, namely finding whether $\bD$-equivalence implies $\bL$-equivalence. This was studied in \cite[Theorem 1.5]{IMOU16b} and \cite[Theorem 3.1]{EF} where counterexamples in the context of abelian varieties were found. Still, up to this moment no counterexample was found among simply connected varieties.

\begin{problem}[Problem 1.2, \cite{IMOU16b}; Conjecture 1.6, \cite{KS}]\label{Conj:KS}
Does $\bD$-equivalence implies $\bL$-equivalence for simply connected varieties?
\end{problem}

Non-trivial examples of this implication are provided by pairs of $\bD$-equivalent and $\bL$-equivalent varieties $X$ and $Y$ which are not stably birational equivalent, so that $[X] \neq [Y]$, see \cite[Corollary 2.6]{LaLu}. 
The first work, in a slightly weaker setting than $\bL$-equivalence, is given by the Pfaffian-Grassmannian Calabi-Yau threefolds of \cite{Bor}, while the very first non-trivial example is given in \cite{Mar}.

Other examples, involving for instance pairs of K3 surfaces, were studied in \cite{IMOU16a}, \cite{KS}, \cite{HL}, \cite{KapRam}, \cite{OR} and \cite{Manivel}. For an overview on this problem please consult \cite{IMOU16b}.

In this work we study another class of examples of derived equivalent K3 surfaces by means of Verra's fourfolds. These are the most elementary pairs of non birational and $\bL$ equivalent $K3's$.
We will work over the field of complex numbers $\bC$.

\begin{defn} \label{defn:verrafourfold}
A Verra fourfold $V_4$ is smooth fourfold obtained as a double cover $V_4 \to \bP^2 \times \bP^2$ branched along a divisor $V_3$ of bidegree $(2,2)$.
\end{defn}

Such a variety $V_3$ is also called a Verra threefold; it was studied by Verra in the paper \cite{Ver04}. A fourfold $V_4$ naturally comes with two projections $p_i: V_4 \to \bP^2$, $i=1,2$ which are families of two dimensional quadrics. These give rise to two K3 surfaces $S_i$, namely the double coverings of $\bP^2$ ramified along the sextic curves discriminant locus of the $p_i$. We will call them the K3 surfaces (geometrically) associated with $V_4$. These two K3 surfaces are equipped with Brauer classes $\beta_i\in H^2(S_i, \mathcal{O}^*_{S_i})$ given by the Brauer-Severi varieties, being the two components of the Hilbert scheme of lines of the Verra fourfold $V_4$. It is easy to see that the K3 surfaces $S_1$ and $S_2$ are always twisted derived equivalent, meaning that the categories $\bD^b(S_i,\beta_i)$ are equivalent. See \cite{CaldararuThesis} for a detailed introduction to twisted derived categories.
It has been proven in \cite{IKKR} that $V_4$ has an associated hyperk\"ahler fourfold $X$. 
This $X$ is related to $V_4$ and $S_i$ in the following ways:
\begin{itemize}
\item $X$ is the base of a $\bP^1$-fibration on the Hilbert scheme of $(1,1)$ conics on $V_4$.
\item $X$ is birational to the moduli spaces $M_v(S_i,\beta_i)$ of twisted sheaves with Mukai vector $(0,H_i,0)$ on $S_i$, for both $i=1,2$.
\end{itemize}

In this paper, we investigate the impact of these relations on the primitive cohomology spaces of $V_4$, $X$ and $S_i$. 
We are mainly interested in the case where the twists $\beta_i$ are trivial. 
In this case, we know that both $\bD^b(S_i)$ can be identified with two equivalent subcategories of $\bD^b(V_4)$. This means that the two K3 are $\bD$-equivalent. 
According to \cite[Section 2.6.2]{KS}, the surfaces $S_1$ and $S_2$ are also $\bL$-equivalent. 
Indeed, when the Brauer classes of the projections $\pi_i$ vanish we get the following relation in the Grothendieck ring of varieties over $\bC$:
$$[X]=[\bP^2]\cdot (1+\bL^2)+[S_i]\cdot \bL.$$
 These relations together give
\begin{equation*} 
\bigl([S_1] - [S_2]\bigr)\cdot \bL = 0.
\end{equation*}
The aim of this work is to answer a question of \cite{KS}, by providing examples of Verra fourfold with associated $\bL$-equivalent K3 surfaces being smooth, non-isomorphic, and with trivial Brauer classes.

Our main result is Theorem \ref{Thm:theorem non isom}, where we prove that for a very general $V_4$ in any $18$ dimensional Noether-Lefschetz type family of Verra fourfolds (i.e. a family defined by the condition of containing a fixed sublattice in its Picard lattice), the K3 surfaces $S_i$ are not isomorphic. 
This will lead to an $18$ dimensional family of examples of the situation described in \cite{KS}. 
Our theorem can be compared to the results of Nikulin and Madonna; they considered pairs of K3 surfaces of degrees $8$ and $2$ associated to a net of quadrics in $\mathbb{P}^5$, see \cite{NM1, NM2, CynkRams}. Furthermore, we propose a method to construct explicit examples of derived equivalent hyperk\"ahler manifolds which should provide a higher dimensional evidence for Problem \ref{Conj:KS}.

\medskip

\textbf{Plan of the paper.} Some basic notions concerning lattice theory and Verra fourfolds are recalled in Secion \ref{sec:preliminaries}. Section \ref{Sec: K3 surfaces I} is devoted to the study of the Picard lattice of $X$, and to the proof of the main Theorem \ref{Thm:theorem non isom}. In the last part of this section, we compare the primitive cohomology of $X$, $V_4$, and $S_i$ providing an Abel-Jacobi map between $V_4$ and $X$.
The explicit construction of the family $\sF$ of Verra fourfolds mentioned in the introduction takes place in Section \ref{Sec:Construction Family}, where we also provide an explicit example of a smooth Verra fourfold $V_4$ in our family by means of Macaulay2. Finally, Section \ref{Sec:Plans} contains some comparison with the case of cubic fourfolds and some possible developments in the case of double EPW sextics.

\section{Background and notations}  \label{sec:preliminaries}
This section is devoted to a short recap of the basic notions we will need in the paper. We remark that a lot of the background is also covered in \cite{IKKR} and in \cite{CKKM}. The interested reader will find there other viewpoints concerning the geometry of Verra fourfold and their relations with hyperk\"ahler manifolds. \medskip

\textbf{Lattices and cohomolgy. } A lattice is a free module over $\bZ$, together with a bilinear form with integer values. A good reference are the works of Nikulin \cite{N1, N2}. Consider a K3 surface $S$, i.e. a surface with trivial canonical bundle and irregularity equal to zero. We can associate to $S$ a weight-two Hodge structure given by $H^{2*}(S)$; its integer part is denoted by $\Lambda_S$ and, once endowed with the intersection pairing, has the structure of a lattice, called the Mukai lattice of $S$. An element $\eta$ in $\Lambda_S$ is called a Mukai vector; in this setting $\eta=(\eta_0,\eta_1,\eta_2)$ where $\eta_i \in H^{2i}(S,\bZ)$. The Mukai lattice of a K3 surface is isomorphic to $3U \oplus 2E_8(-1)$, where $U$ is the standard hyperbolic lattice and $E_8$ is the root lattice associated to the corresponding Dynkin diagram. We will also define the Mukai lattice associated to a hyperk\"ahler fourfold $X$, that is a projective irreducible holomorphic symplectic manifold, see \cite{B}. It is defined in a similar way as for the K3, and it is explicitly given by $4U \oplus 2E_8(-1)$, see \cite{CKKM}. \medskip

\textbf{Twisted varieties.} Let $X$ be an algebraic variety and consider an element $\alpha$ of $H^2_{\text{{\'e}t}}(X,\sO^*_X)$, in the \'etale topology. This cohomology group is isomorphic to the Brauer group $\br(X)$, whose elements are called Azumaya algebras. A pair $(X, \alpha)$ is called noncommutative variety; we refer to \cite[Appendix D]{KuzHyp} for a general introduction. Once we fix an element $\alpha$ as above, we can use it to construct the category $\coh(X, \alpha)$ of $\alpha$-twisted coherent sheaves on $X$. If we think $\alpha$ as an Azumaya algebra, the objects of $\coh(X, \alpha)$ are coherent sheaves of $\alpha$-modules. If we see $\alpha$ as a $2$-cocycle $\alpha_{\text{ijk}}$ in $H^2_{\text{{\'e}t}}(X,\sO^*_X)$, a twisted coherent sheaf is given by pairs $(E_\text{i},\phi_{\text{ij}})$, where $E_i$ is a coherent sheaf on $\sU_i$, where $X=\cup \sU_i$, and the $\phi_{\text{ij}}$'s satisfy the conditions $\phi_{\text{ii}}=\id$, $\phi_{\text{ji}}=\phi_{\text{ij}}^{-1}$ and $\phi_{\text{ij}} \circ \phi_{\text{jk}} \circ \phi_{\text{ki}} = \alpha_{\text{ijk}} \id$.
The category $\coh(X, \alpha)$ is abelian, so it make sense to consider the twisted derived category $\bD^b(X,\alpha)$. When $\alpha$ is trivial we get back the ordinary category of coherent sheaves and its derived category. If $S$ is a K3 surface, there is an alternative description of an element $\alpha$ in $\br(S)$ by means of lattice theory. By \cite[Section 2.1]{MR2166182}, we have $\br(S) \cong \hom(T_S,\bQ/\bZ)$. Moreover, an element $\alpha$ in $\br(S)$ can be identified with a surjective homomorphism from $T_S$ to ${\bZ}/{n \bZ}$, where $n$ is the order of $\alpha$. We can use the kernel of this morphism to lift $\alpha$ to an element $B \in H^2(S,\bQ)$. The element $B$ is called the lift of $\alpha$, and it is determined up to elements in $\Pic(S)/n$. In this paper we will only consider K3 surfaces twisted with certain $\alpha$ constructed starting from Verra fourfolds. \medskip

\textbf{Verra fourfolds. }
A Verra fourfold $V_4$, see Definition \ref{defn:verrafourfold}, has two structures of quadric fibration determined by the two projections to $\bP^2$. These structure provides many nice geometrical properties. In particular, they give the possibility to associate two K3 surfaces to $V_4$, as explained in the introduction, see \cite[Chapter I]{BPrym} for more details on this construction. These quadric fibrations have also associated Brauer classes, which allows us to consider the associated K3 surfaces as twisted varieties, see \cite{KuzQ}. Note that $V_4$ can be seen as an intersection of a cone over $\bP^2 \times \bP^2$ with a quadric hypersurface. The hyperplane section in this embedding provides a natural polarization of $V_4$. All together, starting from $V_4$ we get two associated twisted K3 surfaces $(S_i, H_i, \alpha_i)$, where $H_i$ is the polarization and $\alpha_i$ is the twist.
It was proven in \cite{MR2166182} that Verra fourfolds, together with cubic fourfolds containing a plane and smooth complete intersection of three quadrics in $\bP^5$ are the three geometric realisations of $2$-torsion elements in the Brauer group of a general K3 surface of degree 2. The moduli space of Verra fourfolds has dimension $19$, see \cite[Section 0.3]{IKKR}. If we select a very general Verra fourfold, the associated K3 surfaces $S_i$ have degree two, Picard group of rank one, and the Brauer classes  $\beta_i$ of $S_i$ are non trivial of order 2. By \cite[Proposition 2.1]{MR1738215} we get that requiring that one of the $\beta_i$ vanishes gives a Noether-Lefschetz divisors in the moduli space. This means that we get a fixed element different from the hyperplane class in the Picard lattice of the $S_i$. It is then clear that the Picard rank of $S_i$ is greater than or equal to two, being exactly two if we pick a very general Verra fourfold in this divisor class. Notice that $V_4$ can also be reconstructed from $(S_i,\beta_i)$, see \cite[Section 9.8]{MR2166182}.

The notion of twisted sheaves provides a link between Verra fourfolds and hyperk\"ahler fourfold, via the construction of moduli space. We refer to \cite{Yoshioka} for all the basic notions about such moduli spaces, and in particular about the existence of the moduli space of twisted sheaves we need in this paper. By \cite[Theorem 5.1]{CKKM}, there is an hyperk\"ahler fourfold $X$ associated to $V_4$, which is birational to the moduli spaces of twisted sheaves
$$M_{\eta_1}(S_1,\alpha_1) \stackrel{\operatorname{bir}}{\cong} M_{\eta_2}(S_2,\alpha_2),$$
where $\eta_i$ are the Mukai vectors $(0,H_i,0)$. \medskip

\textbf{Notations and remarks.} 
Let $Y$ be a variety of dimension $n$. We denote by $H^k_0(Y,\bZ)$ the $k$-th primitive cohomology group of $Y$ with coefficients in $\bZ$, namely the kernel of the map $L^{n-k+1}: H^k(Y) \to H^{2n-k+2}(Y)$ obtained by taking the exterior product with $\omega^{n-k+1}$, see \cite{Voisin}.
When we say that a variety is very general we mean that we choose it away from a countable union of divisors of the appropriate moduli space. We will work over the field of complex numbers $\bC$.

\section{K3 surfaces associated to families of Verra fourfolds} \label{Sec: K3 surfaces I}
Let us denote by $(S_1,H_1, \alpha_1)$ and $(S_2,H_2, \alpha_2)$ the two K3 surfaces of degree $2$ associated to a Verra fourfold $V_4$, with polarization $H_i$ and twist $\alpha_i$. In order to prove the main theorem, we will identify them in terms of their periods and Brauer classes. In the first part of this section we will construct, by means of lattice theory and Torelli theorem, two K3 surfaces $\tilde{S}_i$. Corollary \ref{all smooth} tells us that $\{S_1,S_2\}$ and $\{\tilde{S}_1,\tilde{S}_2\}$ coincide up to isomorphisms. This will be crucial in the second part for the proof of Theorem \ref{Thm:theorem non isom}, since we will be able to work on the $\tilde{S}_i$ instead of the $S_i$. The following lemma is the first step towards the proof that $S_1$ and $S_2$ are uniquely determined up to isomorphism.

\begin{lem} \label{only two K3} 
Let $X$ be the hyperk\"ahler fourfold associated to a very general Verra fourfold. There are at most two isomorphism types of polarized twisted K3 surfaces $(S, H, \beta)$ satisfying the following properties:
\begin{enumerate}
\item $H^2=2$;
\item $X$ is birational to $M_{\eta}(S,\beta)$, where $\eta$ is the Mukai vector $(0,H,0)$;
\item $\beta$ is non-trivial of order 2 and admits a lift $B$ with $B^2=B\cdot H=0$.
\end{enumerate}
\end{lem}
\begin{proof}
Let $\tilde{\Lambda}$ be the extended Mukai lattice of $X$, and assume that the twist $\beta$ admits a lift $B\in \frac{1}{2} H^2(X,\bZ)$ such that $B\cdot H=0$ and $B^2=0$. Notice that the Brauer class has order two because we have a $\bP^1$ bundle as Brauer-Severi variety, see \cite[Section 9.8]{MR2166182}, \cite[Section 1]{Yoshioka}. Let $\sigma \in H^2(S,\bC)$ be the period of $S$ and consider the twisted period 
$$\sigma_B=\sigma + \sigma \wedge B \in \tilde{\Lambda} \otimes \bC.$$  
It was proven in \cite[Theorem 3.19]{Yoshioka} that $H^2(X,\mathbb{Z})$ is perpendicular to the Mukai vector $\eta$ in $\tilde{\Lambda}$. Moreover, the twisted period $\sigma_B$ belongs to $\eta^{\perp}\otimes \bC$ and coincides with the period $x$ of $X$. Nevertheless, in order to avoid possible misunderstandings, we will keep naming them with their two different names.

The lattice $\Pic(S,B)=(\sigma_B)^{\perp}\subset \tilde{\Lambda}$ was studied in \cite{MSFano}, see in particular Lemma 3.1 there, and it can be written as
$$\Pic(S,B)=\bigl\langle(0,H,0), (0,0,1), (2,2B,0)\bigr\rangle.$$ 
Similarly, $\Pic(X)$ corresponds to $x^{\perp}\cap H^2(X,\mathbb{Z})$ and hence also to $\sigma_B^{\perp}\cap \eta^{\perp}\subset \tilde{\Lambda}$. It follows that $\Pic(X)$ in $\tilde{\Lambda}$ is generated by 
$$\bigl\langle(0,0,1),(2,2B,0)\bigr\rangle=\bigl\langle(0,0,1),(2,2B,B^2)\bigr\rangle\cong U(2),$$
where we denote by $U$ the standard hyperbolic lattice.
Let us write explicitly 
$$\tilde{\Lambda}=I \oplus J \oplus M \oplus N \oplus 2E_8(-1),$$ where $I, J, M, N$ are four copies of $U$.
Denote their corresponding standard bases in the following way: 
$$I=\langle i_1, i_2 \rangle, \quad J=\langle j_1, j_2 \rangle, \quad M=\langle m_1, m_2\rangle, \quad N=\langle n_1, n_2 \rangle.$$
By Eichler's criterion, see \cite[Theorem 2.9]{BHPV}, there is a unique embedding of $\Pic(S,B)$ in $\tilde{\Lambda}$ up to isometries. We can hence assume that 
$\eta=i_1+i_2$ and $$\Pic(X)=\langle m_1+n_1,m_2+n_2 \rangle.$$ 
As a consequence, the twisted transcendental lattice of $S$ is described as follows
\[
\begin{split}
T_S(B)&=\Pic(S,B)^{\perp}=\langle \eta, \Pic(X) \rangle^\perp = \\
&=\langle i_1-i_2, m_1-n_1,m_2-n_2 \rangle\oplus J\oplus 2E_8(-1)\cong\langle-2\rangle \oplus U(-2)\oplus U\oplus 2E_8(-1).
\end{split}
\]
Recall that the two periods $[\sigma_B]$ and $[x]$ are equal, hence we have $T_S(B)$ is equal to $T_X$, the transcendental lattice of $X$ inside $\eta^{\perp}$.

The projection $\pi: \tilde{\Lambda}\otimes \mathbb{C} \to H^2(S,\mathbb{C})$ induces a lattice embedding $T_X=T_S(B)\to T_S$ whose image is an index $2$ sublattice. We also have $\pi(\sigma_B)=\sigma$.
By \cite[Section 9.8]{MR2166182}, there are exactly two ways for $T_S(B)$ to be embedded in $T_S$ as an order $2$ sublattice. Each of them gives rise to a unique map $$\pi_i:T_S(B)\otimes \bC\to T_S\otimes \bC.$$
It follows that $\sigma=\pi_i(\sigma_B)$ for $i\in \{1,2\}$. Note that the class of the periods 
$$[\sigma_B]=[x] \in \bP( T_S(B)\otimes \bC)= \bP( T_X\otimes \mathbb{C})$$
is uniquely determined by $X$. 

By \cite[Theorem 3.3]{Orlov}, we know that $[\sigma] \in \mathbb{P}(T_S\otimes \mathbb{C})$ determines the derived equivalence class of $S$. Moreover, it has been proved in \cite[Theorem 1.17]{Hosonoetal} that a very general K3 surface of degree $2$ and Picard number $1$ does not admit any Fourier-Mukai partners, hence $[\sigma]\in \mathbb{P}(T_S\otimes \mathbb{C})$ determines $S$ uniquely up to isomorphisms. Since $[\sigma]\in \mathbb{P}(T_S\otimes \mathbb{C})$ is equal to one of the two projections $\pi_i(\sigma_B)$ we have at most two possible K3 surfaces $S$.
\end{proof}

We can now apply Lemma \ref{only two K3} to the hyperk\"ahler fourfold $X$.

\begin{cor}\label{cor: using S not isomorphic}
Let $V_4$ be a very general Verra fourfold with the associated hyperk\"ahler fourfold isomorphic to the moduli space of twisted sheaves on a polarized twisted K3 surface $(S,H,\beta)$, satisfying the assumptions of Lemma \ref{only two K3}. Then we have
$$(S,H,\beta) \cong (S_i,H_i,\alpha_i)$$
for one of the two K3 surfaces $S_i$ associated to $V_4$.
\end{cor}
\begin{proof} 
We can exploit results of \cite[Section 5]{CKKM} and \cite[Section 2]{MR3705236} in order to ensure that the K3 surfaces $S_i$ satisfy the conditions of Lemma \ref{only two K3}. It is then enough to observe that the $S_i$ are not isomorphic in general, and so they cover both the isomorphism classes of Lemma \ref{only two K3}. This will be shown by considering an explicit example, developed at the end of Section \ref{Sec:Construction Family}. 
\end{proof}

These uniqueness results will be fundamental in order to show that the two K3 surfaces that we will construct by using Torelli theorem coincide with the K3 surfaces geometrically obtained from the considered Verra fourfold. Keeping the notation that we introduced in the proof of Lemma \ref{only two K3}, let us consider the following two sublattices of $\tilde{\Lambda}$, both isomorphic to $U$:
\begin{equation*}
U_1=\langle m_1+n_1, m_2 \rangle \quad \text{and} \quad U_2=\langle m_2+n_2, m_1 \rangle.
\end{equation*}
With this choice we get $\tilde{\Lambda}=U_i \oplus \Lambda_{i}$, with $\Lambda_i$ being isometric to the K3 lattice $3U\oplus 2E_8(-1)$. 
The lattice $\Lambda_{i}$ is explicitly given by $\Lambda_{i}=\overline{U_i}\oplus I\oplus J\oplus 2E_8(-1)$, where
\begin{equation*}
\overline{U_1}=\langle n_2-m_2,n_1\rangle \quad \text{and} \quad \overline{U_2}=\langle n_1-m_1,n_2\rangle.
\end{equation*}

We can define new K3 surfaces $\tilde{S}_i$ via Torelli theorem, by using the periods obtained by projecting the element $x$ to $\Lambda_{i}$. In these cases $\eta$ gives a polarization $\tilde{H}_i$ of degree 2 and  $\tilde{B}_i={(m_i-n_i)}/{2}$ determines a Brauer class.

\begin{cor}\label{very general K3 are isomorphic} Let $V_4$ be a very general Verra fourfold. Then, up to isomorphisms, the set of K3 surfaces $\{\tilde{S}_1, \tilde{S}_2\}$ is a subset of the set of K3 surfaces $\{S_1, S_2\}$.
\end{cor}
\begin{proof} The proof follows from the fact that the two K3 surfaces $\tilde{S}_i$ also satisfy the conditions required by Lemma \ref{only two K3}. Indeed, performing the twisted moduli construction we recover a variety $\tilde{X}$ with the same period as $X$. We then apply Torelli theorem for hyperk\"ahler fourfolds to see that $\tilde{X}$ is birational to $X$ giving $(2)$ from Lemma \ref{only two K3}.
\end{proof}

\begin{cor} \label{all smooth} Let $V_4$ be any Verra fourfold such that the surfaces $S_1$ and $S_2$ are smooth.
Then, up to isomorphisms, the set of K3 surfaces $\{\tilde{S}_1, \tilde{S}_2\}$ is a subset of the set of K3 surfaces $\{S_1, S_2\}$.
\end{cor}
\begin{proof}
Let $\phi\colon\mathcal{V}\to \Delta$ be a family of Verra fourfolds over a disc admitting very general Verra fourfolds as fibers and degenerating to $V_4$ over $0$. Considering the relative Hilbert scheme of $(1,0)$ (resp.~$(0,1)$) lines we obtain families of twisted K3 surfaces $ \sS_1$ (resp.~$\sS_2$). 
The relative Hilbert scheme of $(1,1)$ conics gives rise to a family of hyperk\"ahler manifolds described in \cite[Section 0.3]{IKKR}. The latter admits a universal marking which induces a family of periods in the lattice $\tilde{\Lambda}$ which has fixed decompositions $\tilde{\Lambda}=U_i\oplus \Lambda_i$. The projections of these periods to $\Lambda_i$ give rise, via Torelli theorem for K3 surfaces, to a family of $K3$ surfaces $\tilde{ \sS}_i$. 
We infer flat families of $K3$ surfaces $\sS_i\to \Delta$ and
$\tilde{\sS}_i\to \Delta$. After reordering we can assume by Corollary \ref{very general K3 are isomorphic} that $\sS_1\to \Delta$ and
$\tilde{\sS}_1\to \Delta$ (resp.~$\tilde{\sS_2}\to \Delta$ and
$\sS_i\to \Delta$ for $i$ either $1$ or $2$) have isomorphic corresponding 
 very general fibres, and hence the corresponding special fibres are also isomorphic.
\end{proof}

Now we are able to prove our main result: for very general elements of certain families of Verra fourfolds, we will get that the surface $\tilde{S}_1$ and $\tilde{S}_1$ are not isomorphic, and this will force the surface $S_1$ and $S_2$ to be not isomorphic as well. 

\begin{thm} \label{Thm:theorem non isom}
Let $\mathcal{V}$ be a family of Verra fourfolds such that the corresponding families of twisted polarized K3 surfaces $(\sS_i,\sH_i, \sB_i) $ for $i\in \{1,2\}$ have trivial Brauer class. Assume in addition that both families $\sS_i$ are Brill-Noether type families. Then, for a very general element of $\mathcal{V}$, the K3 surfaces $S_1$ and $S_2$ are not isomorphic.
\end{thm}

\begin{proof}
Let $V_4$ be a very general member of the family $\mathcal{V}$. It falls into the hypothesis of Corollary \ref{all smooth} and hence we can compare the two K3 surfaces $\tilde{S}_i$ defined previously. If these are not isomorphic, then the surfaces $S_1$ and $S_2$ are not isomorphic as well by Corollary \ref{all smooth}. Assume that the period on the hyperk\"ahler $X$ is explicitly given by
\begin{equation*}
x = \lambda_1(i_1-i_2) + \delta_1(m_1-n_1) + \delta_2(m_2-n_2) + \ldots,
\end{equation*}
Here and in the following explicit computations we will omit all the coefficients which are not relevant for our computation, writing just ``$\ldots$'' instead.
Note that the Hodge structure on $\tilde{S}_i$ is obtained by projecting the element $x$ to $\Lambda_i$ obtaining, in the case of $\tilde{S}_1$, the element $\sigma_1$, given by
\begin{equation*}
\lambda_{1}(i_1-i_2)-2\delta_1 n_1+\delta_2(n_2-m_2)+\ldots.
\end{equation*}
The Picard lattice of $\tilde{S}_i$ is given by $\sigma_i^\perp \cap \Lambda_i$. Let us assume by genericity of $V_4$ that the Picard rank of the K3 surfaces is equal to $2$. We then have also a second polarization $\tau_i$ on $\tilde{S}_i$. 

The assumption about the twists $\alpha_1$ and $\alpha_2$ being trivial can be translated to numerical conditions on $\tau_i$. Recall that the Brauer class is a map $T_{\tilde{S}_1} \to \bQ/\bZ$ and in our case such a map is given by the multiplication by $1/2(m_1-n_1)$. Up to elements in half of the Picard lattice, we see that this is equivalent to the multiplication by $1/2(n_2-m_2)$. So the class $\alpha_1$ is trivial when $(m_2-n_2) \cdot T_{\tilde{S}_1}$ is even. In order to rephrase this more explicitly, consider an element $v \in T_{\tilde{S}_1}$, it can be written as 
$$v=\epsilon_1(n_2-m_2)+\epsilon_2 n_1 + \ldots. $$
Recall that $T_{\tilde{S}_1}=\eta^\perp \cap \tau_1^\perp$, so we can choose $v \in \Lambda_1$ and then impose the orthogonality conditions. We can see explicitly that $\alpha_1$ is trivial if and only if for every $v \in \Lambda_1$, $\tau_1 \cdot v =0$ implies that the coefficient of the element $n_1$ of the basis is even.

We want to reformulate this condition in terms of the coefficient of the second generator $\tau_1$ of the Picard lattice of $\tilde{S}_1$. Explicitly, let
$$\tau_1=\gamma_1(n_2-m_2)+\gamma_2 n_1 + \ldots.$$
Since the lattice $2U+2E_8(-1)$ is unimodular, we get that $\alpha_1$ is trivial if and only if all the coefficients of $\tau_1$ are even, except $\gamma_1$ being odd. For a fixed $\tau_1$ with this choice of coefficients, consider a very general $\sigma_1$ defining the K3 surface $\tilde{S}_1$ with Picard lattice $\langle \eta, \tau_1 \rangle$. Then we can recover $\sigma_2$:
$$\sigma_2=\lambda_1(i_1-i_2)-\delta_1 n_1+2\delta_2(n_2-m_2)+ \ldots.$$

We can also recover $\tau_2$ from $\tau_1$. Indeed, we know that $\sigma_i\cdot \tau_i$ is independent of $i$, so if $\tau_1$ is determined by the coefficients $(\gamma_1, \gamma_2)$ as above, then $\tau_2$ is given by $(-2\gamma_1, -1/2 \gamma_2)$.

Thanks to the explicit description of the coefficients we carried out previously, we are able to describe a non trivial isomorphism between the transcendental lattices of $\tilde{S}_1$ and $\tilde{S}_2$. We divide by $-2$ the coefficient $\gamma_2$ which we know to be even and multiply by $-2$ the value $\gamma_1$. This sends $\tau_1$ to $\tau_2$ and gives an isomorphism $T_{\tilde{S}_1} \cong T_{\tilde{S}_2}$. Notice that the transcendental lattice is orthogonal to the Picard lattice, so it is saturated. So we have 
$$(T_{\tilde{S}_i} \otimes \bC) \cap H^2(\tilde{S}_i,\bZ) = T_{\tilde{S}_i}.$$
This implies there is a unique extension of the isomorphism to $T_{\tilde{S}_i} \oplus \Pic(\tilde{S}_i)$, and hence on the whole $H^2(\tilde{S}_i,\bC)$. As a consequence such extension does not preserve the groups $H^2(\tilde{S}_i,\bZ)$, since one coordinate is divided by $-2$.

It is a general fact that, once fixed a lattice $L\subset \Lambda_{K3}$ of the K3, the transcendental lattice of a general element of the family of K3 surfaces whose Picard lattice contains $L$ has no nontrivial Hodge automorphisms. As a consequence, for a general element in our family  the isomorphism $T_{\tilde{S}_1} \cong T_{\tilde{S}_2}$ must be unique. If $\tilde{S}_1$ and $\tilde{S}_2$ were isomorphic, then the latter would come from a Hodge isomorphism of the whole lattices $H^2(\tilde{S}_1,\mathbb{Z})$ and $H^2(\tilde{S}_2,\mathbb{Z})$ of the two K3 surfaces. This would be a contradiction, because the isomorphism we have found does not extend. Hence the two transcendental lattices are isometric, but the whole lattices are not.
\end{proof}

\begin{cor} \label{cor: corollary non isom} Let $\mathcal{V}$ be any irreducible $18$ dimensional family of Verra fourfolds whose associated quadric fibrations induce trivial Brauer classes $\sB_i$ on the surfaces from the family $\sS_i$ with $i\in \{1,2\}$, then the surfaces $S_1$ and $S_2$ associated to a very general Verra fourfold of $\mathcal{V}$ are non-isomorphic. 
\end{cor}
\begin{proof} Note that the induced Brauer classes can be trivial only if the K3 surfaces $S_i$ have Picard number at least two. So families $\sS_i$ are 18 dimensional families of polarized K3 surfaces of Picard number greater than or equal to $2$. As a consequence each $\sS_i$ contains (and hence is equal to) a locally complete family of K3 surfaces whose Picard lattices contain a fixed lattice of rank 2 i.e. a Noether Lefschetz type family of dimension 18. It follows that $\mathcal{V}$ satisfies the assumptions of Theorem \ref{Thm:theorem non isom}.
\end{proof}

Notice that this result is slightly different from an analogous result of \cite{KS} made in the context of K3 surfaces of degree 2 with another type of Brauer classes. We proved that for any $18$-dimensional family of Verra fourfold such that the corresponding K3 surfaces have trivial Brauer classes, the K3 surfaces are not isomorphic, while in Theorem 1.9 of \cite{KS} there are some families for which the corresponding surfaces are isomorphic and some families for which they are not.

\subsection{The Abel Jacobi map} \label{Sec:AJ}
All the lattice computations carried out previously could be performed in the cohomology $H^4_0(V_4,\bZ)$ directly, without means of the hyperk\"ahler fourfold $X$. 
Indeed, we can compare the primitive cohomologies $H^2_0(X,\bZ)$ and $H^4_0(V_4,\bZ)$, passing through the primitive cohomology of the K3 surfaces $S_1$ and $S_2$. 

In what follows we will always write $S_i$, meaning that this holds for both $i=1, 2$. We know that $X$ is birational to a moduli space of twisted sheaves on $S_i$. We can then apply \cite[Theorem 3.19]{Yoshioka} in order to obtain a Hodge isometry between $H^2_0(X,\bZ)$ and an index $2$ sublattice, orthogonal to a Mukai vector $v_i$ of $H^2_0(S_i,\mathbb{Z})$. The Brauer class related to $v_i$ corresponds to the Brauer class given by the fibration $p_i:V_4\to \mathbb{P}^2$ by \cite[Proposition 4.1]{IKKR}. 

On the Verra fourfold's side, the quadric fibrations related to $V_4$ can be used to apply \cite[Prop.II.2.3.1]{Laszlo}. We find the following short exact sequence 
$$0 \to H^4_0(V_4,\bZ) \xrightarrow{\Phi} H^2_0(S_i,\bZ) \to \bZ / 2 \bZ \to 0.$$
As a consequence $\Phi$ gives an isomorphism between $H^4_0(V_4,\bZ)$ and an index $2$ sublattice of $H^2_0(S_i,\bZ)$. This time, the sublattice is defined by a Brauer class $\Gamma(V_4)$ which by \cite[Sec. III]{Laszlo} defines a theta characteristic on the discriminant sextic curve associated to $S_i$. By the classification of \cite{MR2166182}, such a theta characteristic corresponds to the Brauer class of the quadric fibration $V_4 \to \mathbb{P}^2$, so we can compare this result with the previous one concerning the hyperk\"ahler $X$.
$$\xymatrix{
              & H^2_0(S_i,\bZ) &\\
H^2_0(X,\bZ)\ar[r]^-{\Theta_v}  &  v^\perp\ar@{^{(}->}[u]  & H^4_0(V_4,\bZ)\ar[l]_-{\Phi}
}$$
We can compose $\Theta_v$ and the inverse of $\Phi$ to obtain a Hodge isomorphism between $H^2_0(X,\bZ)$ and $H^4_0(V_4,\bZ)$.

\section{The family of Verra fourfolds}  \label{Sec:Construction Family}
We want to exploit the results from the previous section in order to provide an example of the situation described in \cite[Section 2.6.2]{KS}.
Assume that $V_4$ is a Verra fourfold such that the Brauer classes related to two K3 surfaces $S_i\to \bP^2$ are trivial. We can then apply \cite[Theorem 2.18]{KS} and deduce the relation
$$\bigl([S_1]-[S_2]\bigr)\bL=0.$$

In order to make the two Brauer classes vanish, we can exploit Proposition 2.1 of \cite{MR1658216}, looking for examples having a zero-cycle of odd degree in both projections $p_i:V_4 \to \bP^2$. The diagonal $\Delta \subset \bP^2 \times \bP^2$ gives rise to an odd cycle of $X$ provided the fact that $V_3$ is totally tangent to $\Delta$, that is $V_3$ restricted to the diagonal is a double conic. Indeed, in that case the preimage of the diagonal splits into a union of two components each being a section of both projections. 

Let us now describe the equations of such threefolds $V_3 \subset \bP^2(x_0, x_1, x_2) \times \bP^2(y_0, y_1, y_2)$. Choose a homogeneous polynomial $q$ of degree $2$ in three variables and three bihomogeneous polynomials $l_i$ of degree $1$ in $(x_0, x_1, x_2)$ and in $(y_0, y_1, y_2)$. Let $V_3$ then be defined by the following equation
$$q(x_0, x_1, x_2)\cdot q(y_0, y_1, y_2) + (x_0y_1-x_1y_0)\cdot l_1 + (x_0y_2-x_2y_0)\cdot l_2 +(x_1y_2-x_2y_1)\cdot l_3=0.$$
Let us now compute the dimension of the space parameterizing such Verra threefolds, and therefore the Verra fourfolds we are interested in. The polynomial $q$ counts for $6$ parameters, and together the linear form $l_i$ count for $21$, that is $3 \cdot 9$ minus $6$ relations. Removing the number of automorphisms of $\bP^2 \times \bP^2$ fixing the diagonal we obtain a family of dimension $18$ of such Verra fourfolds. Notice that, up to automorphisms of $\bP^2$, we can assume that $q$ is fixed, being for example the Fermat quadric $x_0^2+x_1^2+x_2^2$.

We described a family $\sF$ of Verra fourfolds of dimension $18$ such that the Brauer classes of $\pi_i$ vanish. We expect a very general member of $\sF$ to give rise to an example of non-isomorphic $\mathbb{L}$-equivalent K3 surfaces. For this only smoothness of the associated K3 surfaces needs to be proven.

\medskip
\textbf{An explicit example.} We construct here an explicit example of a smooth Verra fourfold of the family $\sF$ described before. We will check that the two associated planar sextic curves are smooth as well. 
As a consequence we get that the very general member of $\sF$ does satisfy the smoothness requirements as well, giving us an open subset of an $18$ dimensional family of possible candidates. The following calculations are provided by means of Macaulay2 \cite{M2}. In order to speed up the computation it is possible to use a reduction modulo a (not too small) prime $p$ by changing the base field to $\bZ / p \bZ$.
\begin{m2_code}
-- Base field
BF=QQ;
R1=BF[y_0..y_2];
R2=BF[x_0..x_2];
R=BF[x_0..x_2,y_0..y_2];
-- Construction of the Verra threefold
Q1=((mingens (ideal(x_0,x_1,x_2))^2)*random(R^6,R^1))_0_0;
Q2=sub(Q1, {x_0=>y_0, x_1=>y_1, x_2=> y_2});
for i from 1 to 3 do
	L_i=matrix{{x_0..x_2}}*random(R^3,R^1)*matrix{{y_0..y_2}}*random(R^3,R^1);
F=(Q1*Q2+(x_0*y_1-x_1*y_0)*L_1+(x_0*y_2-x_2*y_0)*L_2+(x_1*y_2-x_2*y_1)*L_3)_0_0;
-- Check that it is smooth
dim saturate ((ideal (singularLocus ideal F)),(ideal(x_0..x_2)*ideal(y_0..y_2)))
\end{m2_code}

\medskip
\textbf{The corresponding sextic curves.}
In order to construct the example we need the sextic curves associated to the Verra fourfold to be smooth such that they give rise to K3 surfaces without blowing up singularities. The following code requires the previous one to be executed.

\begin{m2_code}
-- Construction of the first sextic
for i from 0 to 5 do 
	V1_i=((coefficients(F,Variables=>{x_0,x_1,x_2}))_1)^{i}_0_0;
S1=(map(R1,R)) det matrix{{2*V1_0,V1_1,V1_3},{V1_1,2*V1_2,V1_4},{V1_3,V1_4,2*V1_5}};
dim saturate ideal singularLocus ideal S1
 
-- Construction of the second sextic
for i from 0 to 5 do 
	V2_i=((coefficients(F,Variables=>{y_0,y_1,y_2}))_1)^{i}_0_0;
S2=(map(R2,R)) det matrix{{2*V2_0,V2_1,V2_3},{V2_1,2*V2_2,V2_4},{V2_3,V2_4,2*V2_5}};
dim saturate ideal singularLocus ideal S2
\end{m2_code}

\begin{prop}
There exists an $18$-dimensional family $\sF$ of Verra fourfolds such that the very general $V_4$ in $\sF$ has the following properties:
\begin{itemize}
\item The K3 surfaces $S_i$ associated to $V_4$ are double cover of $\bP^2$ ramified over smooth sextics;
\item The Brauer classes related to the $S_i$ are trivial.
\item $S_1$ and $S_2$ are not isomorphic. As a consequence, the class $[S_1]-[S_2]$ is not trivial in the Grothendieck ring.
\end{itemize}
\end{prop}
\begin{proof} 
The previous constructions and computations show that the first two properties hold. The third property is non-trivial, and it is a consequence of Corollary \ref{cor: corollary non isom}.
\end{proof}
To complete the picture we take the opportunity to use our calculations above to provide an explicit example in which  the sextic curves are not projectively isomorphic, which implies that the polarized K3 surfaces $(S_1,H_1)$ and $(S_2,H_2)$ are not equivalent and completes the proof of Corollary \ref{cor: using S not isomorphic}. The following code requires the previous one to be executed.

\begin{m2_code}
-- Check that the two sextics are not isomorphic
Rt=BF[z_1..z_9][x_0..x_2,y_0..y_2];
FS=map(Rt,R1, transpose sub (matrix{{z_1,z_2,z_3},{z_4,z_5,z_6},{z_7,z_8,z_9}}*
	transpose matrix{{y_0,y_1,y_2}}, {y_1=>x_1,y_2=>x_2,y_0=>x_0}));
fS1=FS(S1);
PT=((map(Rt,R1))(S1))-fS1;
SFD=(coefficients(PT))_1;
IDE=(ideal SFD)+(ideal(1-det(matrix{{z_1,z_2,z_3},{z_4,z_5,z_6},{z_7,z_8,z_9}})));
dim IDE
\end{m2_code}

\section{Comparison with the cubic fourfold case and more} \label{Sec:Plans}
Cubic fourfolds are hypersurfaces in $\bP^5$ of degree $3$. Their Hodge theory and derived category have been extensively studied over the years. A Torelli-type theorem for smooth cubic fourfolds was proved in \cite{MR860684}, and the image of the period map was studied in \cite{MR2507640}. The moduli space of cubic fourfolds was studied in \cite{MR1738215}, where the divisors $\sC_d$ parametrizing special cubic fourfolds of discriminant $d$ are described. Recall that a cubic fourfold is special if it contains a surface not homologous to a complete intersection. In this section we are interested in cubic fourfolds containing a plane. They belong to the divisor $\sC_8$, see \cite[Section 4.1.1]{MR1738215}. The derived category of a cubic fourfold was studied in \cite{MR2605171}. More recently, deep relations between special cubic fourfolds and twisted K3 surfaces were studied in \cite{MR3705236}. There are many possible comparisons between Verra fourfolds and cubic fourfolds. It is well known (see \cite{MR818549}) that the Fano variety of lines of a cubic fourfold is a hyperk\"ahler manifold. Moreover, it has been proved in \cite{MR2166182} that both Verra fourfolds and cubic fourfolds in $\sC_8$ are geometrical realizations of $2$-torsion elements of the Brauer group of a K3 surface. We want to capture a similar situation as for the Verra fourfolds, where we can obtain two K3 surfaces by exploiting the two projections to $\bP^2$.

One possibility is to study the family of cubic fourfolds with two nodes, which has dimension $18$ in the moduli space of all cubic fourfolds. The results of \cite{MR2605171} guarantee that there are two untwisted K3 surfaces associated to such a cubic fourfold, and it is easy to prove that they are also $\bL$-equivalent. In fact, it is possible to show that the two singular K3 surfaces are also isomorphic, giving then a trivial $\bL$-equivalence.

Another instance of the problem is provided by cubic fourfolds containing two planes, again a family of dimension $18$ in the moduli space of all cubic fourfolds, consisting of two components depending on the planes being disjoint or not. The first case has been studied in \cite[Example 5.9]{GS}. Let $Y$ be a general cubic fourfold containing two disjoint planes. The projection from each of these planes provides a quadric fibration over $\mathbb{P}^2$ whose discriminant locus is a sextic curve and for which the second plane is a section. We thus have two K3 surfaces associated to $Y$. The two K3 surfaces are again derived equivalent and $\mathbb{L}$-equivalent. But by \cite{MR860684} the K3 surfaces obtained in this case are also isomorphic, which makes these equivalences trivially fulfilled. 

\begin{question}
Assume $Y$ to be a cubic fourfold in $\sC_8$ containing two planes intersecting in one point. Are the two associated K3 surfaces isomorphic?
\end{question}

A different context in which we can find several K3 surfaces associated to one variety is the case of double EPW sextics, which has been studied in \cite{Ogrady}.
Let $A$ be a lagrangian space in $\wedge^3\bC^6$ with respect to a fixed symplectic form. Let $\bG(3,6)$ be the Grassmannian of $3$-spaces in $\bP^6$. Suppose that $\bP(A)\cap \bG(3,6)$ is a finite set such that $T_{P_i} \cap \bP(A)=P_i$ for $i=1,\dots, k$, where $T_{P_i}$ is the projective tangent space to $\bG(3,6)$ at $P_i$.
In the forthcoming paper \cite{GKV}, the first author and Verra construct in this context $k-1$ nodal Verra fourfolds $V_j$ with associated K3 surfaces $S_{j,1}$ and $S_{j,2}$. It is also possible to prove that all the surfaces $S_{j,i}$ are twisted derived equivalent for every $j$ and $i$. 
This situation suggests in fact the following
\begin{problem} \label{last problem} Consider a hyperk\"ahler variety of dimension $2n$, singular along the union of hyperk\"ahler manifolds of dimension $2(n-1)$.
Are all such submanifolds twisted $\bD$-equivalent?
\end{problem}
If the number of components of the singular locus of those hyperk\"ahler varieties is big enough, we expect they are in fact $\bD$-equivalent.

In the case of nodal Verra fourfolds above we expect that the situation when the K3 surfaces are $\bD$-equivalent happens when $k\geq 16$. It should be still possible to check that such K3's are $\bL$-equivalent as well. Again one has to face the problem of the non-triviality of the examples, which is now more difficult due to the fact that the fourfolds $V_j$ are nodal. Nevertheless, we expect some of the K3 surfaces to be non-isomorphic.

Further cases to check in the context of Problem \ref{last problem} and resulting equivalences are hyperk\"ahler eightfolds constructed in \cite{LLSS} induced by a two nodal cubic.
Moreover, the EPW cubes induced by a Lagrangian $A\subset \wedge^3\mathbb{C}^6$ intersecting the Grassmanian transversally in a finite number of points as above also provide candidates for $\bD$-equivalent varieties of dimension $4$. Hopefully, these examples could provide some new background for the study of Problem \ref{Conj:KS} in dimension higher than $2$.

\section*{Acknowledgements}
G.K. is supported by the project 2013/08/A/ST1/00312. M.K. is supported by the project NCN 2013/10/E/ST1/00688,
R.M. is supported by the Department of Mathematics and Natural Sciences of the University of Stavanger in the framework of the grant 230986 of the Research Council of Norway and  by Firb 2012 Moduli spaces and applications granted by Miur. It is a pleasure to thank Sergey Galkin, Atanas Iliev, Giovanni Mongardi and Kristian Ranestad for their useful suggestions and helpful discussions. The authors are especially grateful to Alexander Kuznetsov for his valuable comments and suggestions he gave us and to Atsushi Ito, Makoto Miura, Shinnosuke Okawa, and Kazushi Ueda for their comment on a preliminary version of this work. We want also to thank an anonymous referee for the careful reading of the manuscript and helpful comments and suggestions.


\end{document}